\documentclass[reqno]{amsart}
\usepackage{amsthm}
\usepackage{amsmath}

\numberwithin{equation}{section}

\theoremstyle{plain}
\newtheorem{thm}{Theorem}[section]
\newtheorem{prop}[thm]{Proposition}

\newtheorem{lem}[thm]{Lemma}
\theoremstyle{remark}
\newtheorem{rem}[thm]{Remark}
\theoremstyle{definition}
\newtheorem{exmp}[thm]{Example}

\numberwithin{equation}{section}
\email{inoue@math.sci.hokudai.ac.jp}
\email{y-kasa@math.sci.hokudai.ac.jp}
\email{punamphartyal@gmail.com}

\begin{document}

\renewcommand{\thefootnote}{\fnsymbol{footnote}}
\footnote[0]{2000\textit{ Mathematics Subject Classification}
Primary 60G25; Secondary 60G15.}

\date{January 17, 2008}
\title[\null]{Baxter's inequality for 
fractional Brownian\\
motion-type processes with Hurst index\\
less than $1/2$}
\author[\null]{Akihiko Inoue, Yukio Kasahara and Punam Phartyal}
\address{Department of Mathematics \\
Faculty of Science \\
Hokkaido University \\
Sapporo 060-0810 \\
Japan}


\begin{abstract}
We prove an analogue of Baxter's
inequality for fractional Brownian motion-type processes with
Hurst index less than $1/2$.
This inequality is concerned with the norm estimate of the
difference between finite- and infinite-past predictor
coefficients.
\end{abstract}

\maketitle


\section{Introduction}
\label{sec:1}

To explain Baxter's inequality 
in the classical setup, 
we consider a centered, weakly stationary 
process $(X_k: k\in\mathbf{Z})$, and 
write $\phi_{j}$ and $\phi_{j,n}$ for the infinite- and finite-past 
predictor coefficients, respectively:
\begin{equation}
P_{(-\infty,-1]}X_0=\sum\nolimits_{j=1}^{\infty}\phi_{j}X_{-j},
\qquad
P_{[-n,-1]}X_0=\sum\nolimits_{j=1}^n\phi_{j,n}X_{-j},
\label{eq:1.1}
\end{equation}
where $P_{(-\infty,-1]}X_0$ and $P_{[-n,-1]}X_0$ denote 
the linear least-squares predictors of $X_0$ based on the observed values 
$\{X_{-j}: j=1,2,\dots\}$ and $\{X_{-j}: j=1,\dots,n\}$, respectively. 
There are many models in which $\phi_{j,n}$'s are difficult to compute exactly 
while the computation of $\phi_j$'s are relatively easy. 
In fact, this is usually so for the models with explicit 
spectral density. 
It is known that $\lim_{n\to\infty}\phi_{j,n}=\phi_n$ 
(see, e.g., Pourahmadi, 2001, Theorem 7.14). 
Therefore, it would be natural to approximate 
$P_{[-n,-1]}X_0$ replacing the finite-past predictor coefficients 
$\phi_{j,n}$ by the infinite counterparts 
$\phi_{j}$. 
Then the error can be estimated by
\begin{equation}
\left\Vert P_{[-n,-1]}X_0
-\sum\nolimits_{j=1}^{n}\phi_{j}X_{-j}\right\Vert
\le\Vert X_0\Vert\sum\nolimits_{j=1}^n\vert \phi_{j,n} - \phi_j\vert,
\label{eq:1.2}
\end{equation}
where $\Vert Z\Vert:=E[Z^2]^{1/2}$. 
The question thus arises of estimating the right-hand side of
(\ref{eq:1.2}).
Baxter (1962) showed that for short memory processes, there exists
a positive constant $M$ such that
\[
\sum\nolimits_{j=1}^n\vert \phi_{j,n} - \phi_j\vert
\le M\sum\nolimits_{k=n+1}^{\infty}\vert \phi_k\vert
\qquad\mbox{for all $n=1,2,\dots$}.
\]
This {\it Baxter's inequality} was extended to long memory
processes by  Inoue and Kasahara (2006). 
See also Berk (1974), Cheng and Pourahmadi (1993) and 
Pourahmadi (2001, Section 7.6.2).

In Inoue and Anh (2007), prediction formulas similar to
(\ref{eq:1.1}) were proved for a class of
continuous-time, centered, stationary-increment, Gaussian
processes $(X(t): t\in\mathbf{R})$ that includes fractional
Brownian motion $(B_H(t): t\in\mathbf{R})$ with Hurst index
$H\in (0,1/2)$ (see Section 3 for the definition). For
\begin{equation}
-\infty<t_0\le 0\le t_1<t_2<\infty,
\qquad
t_0<t_1,
\qquad T:=t_2-t_1,\qquad t:=t_1-t_0,
\label{eq:1.3}
\end{equation}
the prediction formulas take the following forms:
\begin{equation}
P_{(-\infty ,t_1]}X(t_2)
=\int_0^\infty\psi(s;T)X(t_1-s)ds,
\qquad 
P_{[t_0,t_1]}X(t_2)
=\int_0^t\psi(s;T,t) X(t_1-s)ds,
\label{eq:1.4}
\end{equation}
where 
$P_{(-\infty ,t_{1}]}X(t_2)$ and $P_{[t_{0},t_{1}]}X(t_2)$ are 
the linear least-squares predictors of $X(t_2)$ based on the infinite past 
$\{X(s): -\infty<s\le t_1\}$ and finite past 
$\{X(s): t_0\le s\le t_1\}$, respectively. 

The aim of this paper is to prove an analogue of Baxter's
inequality for $(X(t))$.
Since $\Vert X(s)\Vert$ depends on $s$, a straightforward analogue
of (\ref{eq:1.2}) is not available. 
Instead, we have
\begin{equation*}
\left\Vert P_{[t_0,t_1]}X(t_2) 
-\int_0^t\psi(s;T)X(t_1-s)ds\right\Vert
\le\int_0^t\{\psi(s;T,t)-\psi(s;T)\}\Vert X(t_1-s)\Vert ds.
\end{equation*}
Here $\psi(s;T,t)>\psi(s;T)>0$ 
(see Section 3 below).
We show that there is a positive constant $M$ such that
\[
\int_0^t\{\psi(s;T,t)-\psi(s;T)\}\Vert X(t_1-s)\Vert ds
\le M\int_t^\infty\psi(s;T)\Vert X(t_1-s)\Vert ds 
\quad\mbox{for all $t\ge t_1$},
\eqno{\mathrm{(B)}}
\]
which we call {\it Baxter's inequality for $(X(t))$}. 
To the best of our knowledge, this type of inequality has not been 
demonstrated before. 
The key ingredient in the proof is the representation of the
difference $\psi(s;T,t) - \psi(s;T)$ 
((\ref{eq:3.2}) with Proposition \ref{prop:3.2} below). 
In fact, we prove a general result that includes (B) 
(Theorem \ref{thm:4.2} (b)).


\section{Fractional Brownian motion}
\label{sec:2}

Throughout the paper, we assume $0<H<1/2$. 
We can define the fractional Brownian motion 
$(B_{H}(t):t\in \mathbf{R})$ with Hurst index $H$ 
by the moving-average representation
\[
B_{H}(t)
=\frac{1}{\Gamma (\frac12+H)}\int_{-\infty }^{\infty}
\left\{((t-s)_{+})^{H-\frac12}-((-s)_{+})^{H-\frac12}\right\}
dW(s)
\qquad
(t\in \mathbf{R}),
\]
where $(x)_+:=\max(x,0)$ and 
$(W(t):t\in \mathbf{R})$ is the ordinary Brownian motion. 
In this section, we study the difference
between the finite- and infinite-past predictor coefficients of 
$(B_{H}(t))$.

Let $t_0,t_1,t_2,t$ and $T$ be as in (\ref{eq:1.3}). 
We define the infinite- and finite-past predictors 
$P_{(-\infty ,t_1]}B_H(t_2)$ and $P_{[t_{0},t_{1}]}B_H(t_2)$ 
of $(B_H(t))$, respectively, as we defined in Section 1 for $(X(t))$. 
The following prediction formulas, that is, (\ref{eq:1.4})
for $(B_H(t))$, were established by Yaglom (1955) and Nuzman and
Poor (2000, Theorem 4.4), respectively (see also
Anh and Inoue, 2004, Theorem 1):
\[
P_{(-\infty ,t_1]}B_H(t_2)
=\int_0^\infty\psi_0(s;T)B_H(t_1-s)ds,
\quad 
P_{[t_{0},t_{1}]}B_H(t_2)
=\int_0^t\psi_0(s;T,t)B_H(t_1-s)ds,
\]
where 
\[
\begin{aligned}
&\psi_0(s;T)
=\frac{\cos(\pi H)}{\pi}\,
\frac{1}{s+T}\left(\frac Ts\right)^{\frac12+H}
\qquad(0<s<\infty),\\
&\psi_0(s;T,t)
=\frac{\cos(\pi H)}{\pi}
\left[\frac{1}{s+T}
\left(\frac Ts\right)^{\frac12+H}
\left(\frac{t-s}{t+T}\right)^{\frac12-H}\right.\\
&\qquad\qquad\qquad 
+\left.
\mbox{$(\frac12-H)B_{\frac t{t+T}}(H+\frac12,1-2H)$}\,
\frac1t
\left\{\left(\frac ts\right)
\left(\frac t{t-s}\right)\right\}^{\frac12+H}
\right]
\quad(0<s<t),
\end{aligned}
\]
with $B_s(p,q):=\int_0^su^{p-1}(1-u)^{q-1}du$ being the
incomplete beta function.

Throughout the paper, 
$f(t) \sim g(t)$ as $t\to \infty$ means 
$\lim_{t\to\infty}f(t)/g(t)=1$. 
A positive measurable function $f$, defined on some 
neighbourhood $[M,\infty)$ of $\infty$, is called {\it regularly 
varying with index $\rho\in\mathbf{R}$}, written $f\in R_{\rho}$, 
if for all $\lambda\in (0,\infty)$, 
$\lim_{t\to\infty}f(t\lambda)/f(t)=\lambda^{\rho}$. 
When $\rho=0$, we say that the function is 
{\it slowly varying}. A generic slowly varying function is usually 
denoted by $\ell$. 
See Bingham et al. (1989) for details. 
The function $\Vert B_H(t_1-s)\Vert$ of $s$ is in $R_H$ since
$\Vert B_H(s)\Vert=\vert s\vert^H\Vert B_H(1)\Vert $ 
for $s\in\mathbf R$.

We will use the next lemma in Section 4. 
For $0<H<\frac12$ and $\rho>-\frac12+H$, we put
\[
C(H,\rho)
:=1-\rho\,\mbox{$B(\frac12-H+\rho,\frac12-H)$}\frac{1-2H}{1+2H},
\]
where $B(p,q):=\int_0^1u^{p-1}(1-u)^{q-1}du$ denotes the beta
function.

\begin{lem}\label{lem:2.1}
Let $g$ be locally bounded in 
$[0,\infty)$ and $g\in R_{\rho}$ with $\rho>-\frac12+H$. 
Then, for fixed $T>0$,
\begin{equation*}
\int_0^t\{\psi_0(s;T,t)-\psi_0(s;T)\}g(s)ds
\sim\frac{C(H,\rho)}{\frac12+H-\rho}\cdot t\,\psi_0(t;T)g(t)
\qquad
(t\to\infty).
\end{equation*}
\end{lem}

\begin{proof}
If $t$ is large enough, then $g(t)>0$. For such $t$, we have, 
by simple computation, 
\[
\begin{split}
&\frac1{t\,\psi_0(t;T)g(t)}\int_0^t\{\psi_0(s;T,t)-\psi_0(s;T)\}g(s)ds
=\int_0^1\frac{\psi_0(ts;T,t)-\psi_0(ts;T)}{\psi_0(t,T)}
\frac{g(ts)}{g(t)}ds\\
&\qquad\qquad
=\int_0^1\mathrm{I}(s;T,t)\frac{g(ts)}{g(t)}ds
+\int_0^1
\mathrm{II}(s;T,t)\frac{g(ts)}{g(t)}ds,
\end{split}
\]
where
\[
\begin{split}
&\mathrm{I}(s;T,t)
=s^{-\frac12-H}\frac{1+(T/t)}{s+(T/t)}
\left[\left(\frac{1-s}{1+(T/t)}\right)^{\frac12-H}-1\right],\\
&\mathrm{II}(s;T,t)
=\mbox{$(\frac12-H)B_{\frac t{t+T}}(H+\frac12,1-2H)$}
(t/T)^{\frac12+H}
\{1+(T/t)\}\,\{s(1-s)\}^{-\frac12-H}.
\end{split}
\]
Since 
$B_{t/(t+T)}(H+\tfrac12,1-2H)\sim \left(\tfrac12+H\right)^{-1}
(T/t)^{\frac12+H}$ as $t\to\infty$, 
we easily see that, for $0<s<1$,
\[
\vert\mathrm{I}(s;T,t)\vert\leq\mbox{const.}\times s^{-\frac12-H},
\qquad 
\vert\mathrm{II}(s;T,t)\vert
\leq\mbox{const.}\times\{s(1-s)\}^{-\frac12-H}
\qquad \mbox{($t$ large enough)}.
\]
Put $\delta=\frac12(\frac12-H+\rho)>0$. 
Then, for $0<s<1$, also we have
\begin{equation}
\vert g(ts)/g(t)\vert\leq 2s^{\rho-\delta}
\quad\mbox{($t$ large enough)}
\label{eq:2.1}
\end{equation}
(cf.~Bingham et al., 1989, Theorem 1.5.2).
Therefore, the dominated convergence theorem yields, 
as $t\to\infty$,
\begin{align}
&\int_0^1\mathrm{I}(s;T,t)\frac{g(ts)}{g(t)}ds
\to\int_0^1\frac{(1-s)^{\frac12-H}-1}{s^{\frac32+H-\rho}}ds
=\frac{1-(\frac12-H)B(\frac12-H+\rho,\frac12-H)}{\frac12+H-\rho},
\label{eq:2.2}\\
&\int_0^1\mathrm{II}(s;T,t)\frac{g(ts)}{g(t)}ds
\to\frac{(\frac12-H)B(\frac12-H+\rho,\frac12-H)}{\frac12+H}.
\label{eq:2.3}
\end{align}
In (\ref{eq:2.2}), we have used integration by parts. 
From (\ref{eq:2.2}) and (\ref{eq:2.3}), we obtain the lemma.
\end{proof}

\begin{rem}\label{rem:2.2}
From Lemma \ref{lem:2.1} with $g(t)=\Vert B_H(t_1-t)\Vert$, 
whence $\rho=H$, we see that
\begin{equation*}
\begin{aligned}
\int_0^t\{\psi_0(s;T,t)-\psi_0(s;T)\}\Vert B_H(t_1 - s)\Vert ds
\sim \frac{2}{\pi}\cos(\pi H)C(H,H)T^{\frac{1}{2}+H}
\Vert B_H(1)\Vert&\cdot t^{-\frac{1}{2}}\\
&
(t\to\infty).
\end{aligned}
\end{equation*}
It is interesting that the order of decay here 
is $t^{-1/2}$, whence does not depend on $H$.
\end{rem}


\section{Fractional Brownian motion-type processes}
\label{sec:3}

In this and next sections, we consider the predictor coefficients 
for the fractional Brownian motion-type process 
$(X(t):t\in\mathbf{R})$ in Inoue and Anh (2007). 
It is a stationary-increment Gaussian process defined by
\[
X(t)=\int_{-\infty}^{\infty}\left\{c(t-s)-c(-s)\right\}dW(s),
\qquad
(t\in\mathbf{R}),
\]
where the moving-average coefficient $c$ is a function of the
form
\[
c(t)=\int_{0}^{\infty}e^{-ts}\nu(ds)\quad(t>0),\qquad =0\quad (t\le 0)
\]
with $\nu$ being a Borel measure on $(0,\infty)$ satisfying 
$\int_{0}^{\infty}(1+s)^{-1}\nu(ds)<\infty$. 
We also assume
\begin{align*}
&\lim_{t\to0+}c(t)=\infty,\qquad 
\mbox{$c(t)=O(t^q)\quad (t\to0+)$\quad for some $q>-1/2$,}\\
&c(t)
\sim
\frac{1}{\Gamma(\frac{1}{2}+H)}t^{-(\frac{1}{2}-H)}\ell(t)
\qquad
(t\to\infty),
\end{align*}
where $\ell(\cdot)$ is a slowly varying function and 
$H\in (0,1/2)$.

The process $(X(t))$ also has 
the autoregressive coefficient $a$ 
defined by $a(t):=-(d\alpha/dt)(t)$ for $t>0$, 
where $\alpha$ is the unique function on $(0,\infty)$ satisfying 
\[
-iz\left(\int_{0}^{\infty}e^{izt}c(t)dt\right)
\left(\int_{0}^{\infty}e^{izt}\alpha(t)dt\right)=1\qquad 
(\Im z>0).
\]
We know that $a(t)=\int_{0}^{\infty}e^{-ts}s\mu(ds)$ for 
some Borel measure $\mu$ on $(0,\infty)$ 
(see Inoue and Anh, 2007, Corollary 3.3).
In particular, $a$ is also positive and decreasing on
$(0,\infty)$. 
By Inoue and Anh (2007, (3.12)), we have 
\begin{equation}
a(t)\sim \frac{t^{-(\frac32+H)}}{\ell(t)}
\cdot\frac{(\frac12+H)}{\Gamma(\frac12-H)}
\qquad(t\to\infty).
\label{eq:3.1}
\end{equation}

\begin{exmp}\label{exmp:3.1}
If $\nu$ is given by 
$\nu(ds)=\pi^{-1}\cos(\pi H)s^{-(\frac12+H)}ds$ on $(0,\infty)$, 
then 
$c(t)=t^{-(\frac12-H)}/\Gamma(\frac12+H)$ for $t>0$, 
whence $(X(t))$ reduces to $(B_H(t))$. 
In this case, $a(t)=t^{-(\frac32+H)}(\frac12+H)/\Gamma(\frac12-H)$.
\end{exmp}

We refer to Inoue and Anh (2007, Example 2.6) for another 
example of $(X(t))$ which has two different indexes 
$H_0$ and $H$ describing 
its path properties and 
long-time behaviour, respectively.

We put
\[
b(s,u):=\int_0^uc(u-v)a(s+v)dv
\qquad (s,u>0).
\]
For $k=1,2\ldots$ and $s,t,T>0$, we define $b_k(s,t;T)$
iteratively by
\[
b_1(s;T,t)
:=b(s,T),
\qquad 
b_k(s;T,t)
:=\int_{0}^{\infty}b(s,u)b_{k-1}(t+u;T,t)du 
\qquad (k=2,3,\dots).
\]
Note that $b_k$'s are positive because both $c$ and $a$ are so. 
By Inoue and Anh (2007, Theorems 3.7 and 1.1), 
the infinite- and finite-past predictor coefficients 
$\psi(s;T)$ and $\psi(s;T,t)$ in (\ref{eq:1.4}) are 
given, respectively, by
\[
\begin{aligned}
&\psi(s;T)=b(s,T)=b_1(s;T,t)
\quad
(s>0),\\
&\psi(s;T,t)
=\sum_{k=1}^{\infty}\left\{b_{2k-1}(s;T,t)
+b_{2k}(t-s;T,t)\right\}
\quad(0<s<t).
\end{aligned}
\]
Notice that 
$\psi(s;T,t)$ here corresponds to $h(t-s;T,t)$ in
Inoue and Anh (2007). 
We have
\begin{equation}
\psi(s;T,t)-\psi(s;T)
=\sum_{k=1}^{\infty}
\left\{b_{2k}(t-s;T,t)+b_{2k+1}(s;T,t)\right\}
\qquad(0<s<t),
\label{eq:3.2}
\end{equation}
which plays a key role in the proof of Baxter's inequality (B) 
in the next section. 

To prove Baxter's inequality (B), we need to discuss the following. 
Consider
\[
\beta(t)
:=\int_0^\infty c(v)a(t+v)dv
\qquad
(t>0),
\]
and define $\delta_k(t,u,v)$ for $k=1,2,3,\ldots$ and $t,u,v>0$,
iteratively by
\[
\delta_1(t,u,v)
:=\beta(t+u+v),
\qquad
\delta_k(t,u,v)
:=\int_0^\infty\beta(t+v+w)\delta_{k-1}(t,u,w)dw
\qquad(k=2,3,\ldots).
\]

\begin{prop}\label{prop:3.2}
\label{prop:5.2} For $s,t,T>0$ and $k\ge2$,
\[
b_k(s;T,t)
=\int_0^Tc(T-v)dv\int_{0}^{\infty}a(s+u)\delta_{k-1}(t,u,v)du.
\]
\end{prop}

This can be proved in the same as in 
Inoue and Kasahara (2006, Theorem 2.8); we omit the proof.

Next, we give some results on the asymptotic behaviour of
$\delta_k$'s. 
For $k=1,2,\ldots$ and
$u\ge 0$, we define $f_k(u)$ iteratively by
\[
f_1(u):=\frac1{\pi(1+u)},
\qquad
f_k(u):=\int_0^\infty \frac{f_{k-1}(u+v)}{\pi(1+v)}dv
\qquad(k=2,3,\ldots).
\]

\begin{prop}
\label{prop:3.3}
\begin{enumerate}
\item[\rm (a)]
For $r\in (1,\infty)$, there exists $N>0$ such that 
$0<\delta_k(t,u,v)\le f_k(0)\{r\cos(\pi H)\}^kt^{-1}$ for 
$u,v>0,\ k\in\mathbf{N},\ t\ge N$. 
\item[\rm (b)]
For $k\in\mathbf{N}$ and $u, v>0$,
$\delta_k(t,tu,v)\sim t^{-1}f_k(u)\cos^{k}(\pi H)$
as $t\to\infty$.
\end{enumerate}
\end{prop}

This can be proved in the same as in 
Inoue and Kasahara (2006, Proposition 3.2); we omit the proof.


\section{Baxter's inequality}\label{sec:4}

In this section, we prove Baxter's inequality (B). 
Let $(X(t))$, $\psi(s;T)$ and $\psi(s;T,t)$ be as in Section 3. 
Since $a$ is decreasing, we have 
$a(T+t)\int_0^Tc(v)dv\le\psi(t;T)\le a(t)\int_0^Tc(v)dv$, 
so that (\ref{eq:3.1}) implies
\begin{equation}
\psi(t;T)\sim a(t)\int_0^Tc(v)dv
\sim 
\frac{t^{-(\frac32+H)}}{\ell(t)}
\cdot\frac{(\frac12+H)}{\Gamma(\frac12-H)}\int_0^Tc(v)dv 
\qquad (t\to\infty).
\label{eq:4.1}
\end{equation}

Here is the extension of Lemma \ref{lem:2.1} to $(X(t))$.

\begin{lem}\label{lem:4.1}
Lemma \ref{lem:2.1} with $\psi_0(s;T,t)$ and $\psi_0(s;T)$ 
replaced by $\psi(s;T,t)$ and $\psi(s;T)$, respectively, holds.
\end{lem}

\begin{proof}
For $t$ large enough, 
using (\ref{eq:3.2}), we may write
\[
\begin{aligned}
D(t)&:=\frac1{t\,\psi(t;T)g(t)}
\int_0^t\{\psi(s;T,t)-\psi(s;T)\}g(s)ds
=\int_0^1\frac{\psi(ts;T,t)-\psi(ts;T)}{\psi(t;T)}
\frac{g(ts)}{g(t)}ds\\
&=\sum_{k=1}^{\infty}\int_0^1
\frac{b_{2k}(ts;T,t)}{\psi(t;T)}\frac{g(t(1-s))}{g(t)}ds
+\sum_{k=1}^{\infty}\int_0^1
\frac{b_{2k+1}(ts;T,t)}{\psi(t;T)}\frac{g(ts)}{g(t)}ds
\end{aligned}
\]
and
\[
\begin{split}
\frac{b_k(ts;T,t)}{\psi(t;T)}
&=\frac{a(t)}{\psi(t;T)}
\int_{0}^Tc(T-v)dv\int_{0}^{\infty}\frac{a(ts+u)}{a(t)}\,
\delta_{k-1}(t,u,v)du\\
&=\frac{a(t)}{\psi(t;T)}
\int_{0}^Tc(T-v)dv
\int_{0}^{\infty}\frac{a(t(s+u))}{a(t)}
\cdot
t\,\delta_k(t,tu,v)du.
\end{split}
\]
Put $\delta=\frac13\min\{\frac12-H,\frac12-H+\rho\}>0$.
By (\ref{eq:3.1}), we have $a\in R_{-(3/2)-H}$, and
\[
a(t\lambda)/a(t)
\leq
2\lambda^{-\frac32-H-\delta}
\mbox{ for }
0<\lambda<1,
\quad
\leq 2\lambda^{-\frac32}
\mbox{ for }
\lambda>1
\qquad
(\mbox{$t$ large enough)}
\]
(cf.~Bingham et al., 1989, Theorems 1.5.2 and 1.5.6).
Choose $0<r<1/\cos(\pi H)$ so that
$x:=r\cos(\pi k)\in(0,1)$. 
Then, by Proposition \ref{prop:3.3} (a),
we have for $0<s<1$ and $v>0$,
\[
\begin{split}
\int_{0}^{\infty}\frac{a(t(s+u))}{a(t)}
\cdot
t\,\delta_k(t,tu,v)du
&\leq 2 f_k(0)x^k
\left[\int_0^{1-s}\frac{du}{(s+u)^{\frac32+H+\delta}}
+\int_{1-s}^\infty\frac{du}{(s+u)^{-\frac32}}\right]\\
&\leq 2 f_k(0)x^k
\left[\frac{s^{-\frac12-H-\delta}}{\frac12+H+\delta}
+2\right]
\qquad(\mbox{$t$ large enough)}.
\end{split}
\]
By Inoue and Kasahara (2006, Lemma 3.1), 
$\sum_{k=0}^\infty f_k(0)x^k<\infty$. 
From these facts as well as (\ref{eq:2.1}), (\ref{eq:4.1}), 
Proposition \ref{prop:3.3} (b) and the dominated
convergence theorem, we see that $\lim_{t\to\infty}D(t)=D$, 
where
\[
\begin{split}
D&:=\sum_{k=1}^{\infty}\cos^{2k-1}(\pi H)
\int_0^1
\left\{
\int_0^\infty\frac{f_{2k-1}(u)}{(s+u)^{\frac32+H}}du\right\}
(1-s)^\rho ds\\
&\qquad\qquad\qquad 
+\sum_{k=1}^{\infty}\cos^{2k}(\pi H)
\int_0^1
\left\{\int_0^\infty\frac{f_{2k}(u)}{(s+u)^{\frac32+H}}du\right\}
s^\rho ds.
\end{split}
\]
Since $(B_H(t))$ is a special case of $(X(t))$, 
this also holds for $\psi_0(t;T)$ and $\psi_0(s;T,t)$. 
Therefore, from 
Lemma \ref{lem:2.1}, we conclude that 
$D=C(H,\rho)/(\frac12+H-\rho)$. 
Thus the lemma follows.
\end{proof}

Following theorems are the conclusion of this paper.

\begin{thm}\label{thm:4.2}
Let $g$ be locally bounded in 
$[0,\infty)$ and $g\in R_{\rho}$ with 
$\rho\in(-\frac12+H,\frac12+H)$.
\begin{enumerate}
\item[\rm (a)]
For fixed $T>0$, we have
\[
\int_0^t\{\psi(s,T;t)-\psi(s;T)\}g(s)ds
\sim C(H,\rho)\int_t^\infty\psi(s;T)g(s)ds
\qquad
(t\to\infty).
\]
\item[\rm (b)]
There exists a positive constant $M$ such that
\[
\int_0^t\{\psi(s,T;t)-\psi(s;T)\}g(s)ds
\leq
M\int_t^\infty\psi(s;T)g(s)ds
\qquad
(t>1).
\]
\end{enumerate}
\end{thm}

\begin{proof}
By (\ref{eq:4.1}), the function 
$\psi(s;T)g(s)$ in $s$ belongs to $R_{\rho-\frac{3}{2}-H}$. 
Since $\rho<\frac12+H$, we have
\[
\int_t^\infty\psi(s;T)g(s)ds
\sim \frac1{\frac12+H-\rho}t\psi(t;T)g(t)
\qquad (t\to\infty).
\]
The assertion (a) follows from this and Lemma 4.1,  while (b)
from (a). 
\end{proof}

\begin{thm}\label{thm:4.3}
\begin{enumerate}
\item[{\rm (a)}] 
Baxter's inequality {\rm (B)} holds.
\item[{\rm (b)}]
For fixed $T>0$, we have, as $t\to\infty$,
\[
\int_0^t\{\psi(s,T;t)-\psi(s;T)\}\Vert X(t_1-s)\Vert ds
\sim C(H,H)\frac{1+2H}{\Gamma(\frac12 - H)}
\left(\int_0^Tc(v)dv\right)
\Vert B_H(1)\Vert\cdot t^{-\frac{1}{2}}.
\]
\end{enumerate}
\end{thm}

\begin{proof}
By Inoue and Anh (2007, Lemma 2.7), 
$\Vert X(t)\Vert\sim\Vert B_H(1)\Vert\,t^H\ell(t)$ as 
$t\to\infty$. 
So, (a) follows from Theorem \ref{thm:4.2} (b) 
if we put
$g(s):=\Vert X(t_1-s)\Vert=\Vert X(s-t_1)\Vert$. 
Also, (b) follows from Lemma \ref{lem:4.1} and 
(\ref{eq:4.1}).
\end{proof}




\end{document}